\documentclass[11 pt,leqno]{amsart}
\theoremstyle{definition}
\usepackage{indentfirst, amssymb,amsmath,amsthm}
\usepackage{csquotes}
\usepackage{hyperref}
\newtheorem*{theoA}{Theorem A}
\newtheorem*{theoB}{Theorem B}
\newtheorem*{theoC}{Theorem C}
\newtheorem*{theoD}{Theorem D}
\newtheorem*{theoE}{Theorem E}
\newtheorem*{theoF}{Theorem F}
\newtheorem*{theoG}{Theorem G}
\newtheorem*{theoH}{Theorem H}
\newtheorem*{theoI}{Theorem I}

\newtheorem{ques}{Question}[section]
\newtheorem{theo}{Theorem}[section]
\newtheorem{lem}{Lemma}[section]

\newtheorem{defi}{Definition}[section]
\newtheorem{rem}{Remark}[section]

\newcommand{\ol}{\overline}
\newcommand{\be}{\begin{equation}}
\newcommand{\ee}{\end{equation}}
\newcommand{\beas}{\begin{eqnarray*}}
\newcommand{\eeas}{\end{eqnarray*}}
\newcommand{\bea}{\begin{eqnarray}}
\newcommand{\eea}{\end{eqnarray}}

\numberwithin{equation}{section}
\begin{document}
\title[Uniqueness of meromorphic functions]{Unique range sets of meromorphic functions of non-integer finite order}
\date{}
\author[B. Chakraborty, et al.]{Bikash Chakraborty$^{1}$, Amit Kumar Pal$^{2}$, Sudip Saha$^{3}$ and Jayanta Kamila$^{4}$}
\date{}
\address{$^{1}$Department of Mathematics, Ramakrishna Mission Vivekananda Centenary College, Rahara,
West Bengal 700 118, India.}
\email{bikashchakraborty.math@yahoo.com, bikash@rkmvccrahara.org}
\address{$^{2}$Department of Mathematics, University of Kalyani, Kalyani, West Bengal 741 235, India.}
\email{mail4amitpal@gmail.com}
\address{$^{3}$Department of Mathematics, Ramakrishna Mission Vivekananda Centenary College, Rahara,
West Bengal 700 118, India.}
\email{sudipsaha814@gmail.com}
\address{$^{4}$Department of Mathematics, Ramakrishna Mission Vivekananda Centenary College, Rahara,
West Bengal 700 118, India.}
\email{kamilajayanta@gmail.com}
\maketitle
\let\thefootnote\relax
\footnotetext{Corresponding Author: Bikash Chakraborty}
\begin{abstract} This paper studies the uniqueness of two non-integral finite ordered meromorphic functions with finitely many poles when they share two finite sets. Also, studies an answer to a question posed by Gross for a particular class of meromorphic functions. Moreover, some observations are made on some results due to Sahoo and Karmakar ( Acta Univ. Sapientiae, Mathematica, DOI: 10.2478/ausm-2018-0025) and Sahoo and Sarkar (Bol. Soc. Mat. Mex., DOI: 10.1007/s40590-019-00260-4).
\end{abstract}
\section{Introduction}
We use $M(\mathbb{C})$ to denote the field of all meromorphic functions in $\mathbb{C}$. Also, by $M_{1}(\mathbb{C})$, we denote the class of meromorphic
functions which have finitely many poles in $\mathbb{C}$. The order $\rho( f )$ of the function $f\in M(\mathbb{C})$ is defined as
$$\rho(f)=\limsup\limits_{r\to\infty}\frac{\ln T(r,f)}{\ln r}.$$
Let $S\subset\mathbb{C}\cup\{\infty\}$ be a non-empty set with distinct elements and $f\in M(\mathbb{C})$. We set $$E_{f}(S)=\bigcup\limits_{a\in S}\{z~:~f(z)-a=0\},$$
where a zero of $f-a$ with multiplicity $m$ counts $m$ times in $E_{f}(S)$. Let $\ol{E}_{f}(S)$ denote the collection of distinct elements in $E_{f}(S)$.\par
Let $g\in M(\mathbb{C})$. We say that two functions $f$ and $g$ share the set $S$ CM (resp. IM) if $E_{f}(S)=E_{g}(S)$ (resp. $\overline{E}_{f}(S)=\overline{E}_{g}(S)$).\par
F. Gross (\cite{G2}) first studied the uniqueness problem of meromorphic functions that share distinct sets instead of values. From then, the uniqueness theory of meromorphic functions under set sharing environment has become one of the important branch in the value distribution theory.\par
In 1977, F. Gross (\cite{G2}) proved that there exist three finite sets $S_{j}~(j=1,2,3)$  such that if two non-constant entire functions $f$ and $g$ share them, then $f\equiv g$. In the same paper, he asked the following question:
\begin{ques}
Can one find two (or possibly even one) finite set $S_{j}~(j=1,2)$ such that if two non-constant entire functions $f$ and $g$ share them, then $f\equiv g$?
\end{ques}
In connection to the Gross question, in 1994, H. X. Yi (\cite{2Yi}) proved the following theorem, which answered the Question 1.1 in the affirmative.
\begin{theoA}
Let $S_{1}=\{z~:~z^{n}-1=0\}$ and $S_{2}=\{a\}$, where $n\geq 5$, $a\not=0$ and $a^{2n}\not= 1$. If $f$ and $g$ are entire functions such that $E_{f}(S_{j})= E_{g}(S_{j})$ for $j = 1,2$, then $f\equiv g$.
\end{theoA}
Later in 1998, the same author (\cite{2Yi98}) proved the following theorem:
\begin{theoB}
Let $S_{1}=\{0\}$ and $S_{2} = \{z~:~z^{2}(z+a) - b = 0\}$, where $a$ and $b$ are two non-zero constants such that the algebraic equation $z^{2}(z+a) - b=0$ has no multiple roots. If $f$ and $g$ are two entire functions satisfying $E_{f}(S_{j})= E_{g}(S_{j})$ for $j = 1,2$, then $f\equiv g$.
\end{theoB}
In the same paper (\cite{2Yi98}), the author proved the following theorem also:
\begin{theoC}
If $S_{1}$ and $S_{2}$ are two sets of  finite distinct complex numbers such that any two entire functions $f$ and $g$ satisfying $E_{f}(S_{j})= E_{g}(S_{j})$ for $j = 1,2$, must be identical, then $\max\{\sharp(S_1),\sharp(S_2)\} \geq 3$, where $\sharp(S)$ denotes the cardinality of the set $S$.
\end{theoC}
Thus for the uniqueness of two entire functions when they share two sets, it is clear that the smallest cardinalities of $S_{1}$ and $S_{2}$ are $1$ and $3$ respectively.\par
But for the uniqueness of two meromorphic functions when they share two sets, H. X. Yi (\cite{2Yi}) also answered the question of Gross as follows:
\begin{theoD}
If $S_{1}=\{a+b, a+b\omega,\ldots,a+b\omega^{n-1}\}$ and $S_{2}=\{c_{1},c_{2}\}$ where $\omega=e^{\frac{2\pi i}{n}}$ and $b\not=0$, $c_{1}\not=a$, $c_{2}\not=a$, $(c_{1}-a)^{n}\not=(c_{2}-a)^{n}$, $(c_{k}-a)^{n}(c_{j}-a)^{n}\not=b^{2n}$ $(k,j=1,2)$ are constants. If two non-constant meromorphic functions $f$ and $g$ share $S_{1}$ CM, $S_2$ CM and if $n\geq9$, then $f\equiv g$.
\end{theoD}
In this direction, in 2012, B. Yi and Y. H. Li (\cite{2YL}) gave the following theorem:
\begin{theoE}
If $S_{1}=\{z~:~\frac{(n-1)(n-2)}{2}z^{n}-n(n-2)z^{n-1}+\frac{n(n-1)}{2}z^{n-2}+1=0\}$ where $n\geq5$ is an integer and $S_{2}=\{0,1\}$. If two non-constant meromorphic functions $f$ and $g$ share $S_{1}$ CM, $S_2$ CM, then $f\equiv g$.
\end{theoE}
But recently, J. F. Chen improved Theorems D and E for a particular class of meromorphic functions as:
\begin{theoF} (\cite{2C})
Let $S_{1}=\{\alpha\}$ and $S_{2}=\{\beta_{1}, \beta_{2}\}$, where $\alpha$, $\beta_{1}$, $\beta_{2}$ are distinct finite complex numbers satisfying
$$(\beta_1-\alpha)^{2}\not=(\beta_2-\alpha)^{2}.$$
If two non-constant meromorphic functions $f$ and $g$ in $M_{1}(\mathbb{C})$ share $S_{1}$ CM, $S_2$ IM, and if the order of $f$ is neither an integer nor infinite, then $f\equiv g$.
\end{theoF}
\begin{theoG} (\cite{2C})
Let $S_{1}=\{\alpha\}$ and $S_{2}=\{\beta_{1}, \beta_{2}\}$, where $\alpha$, $\beta_{1}$, $\beta_{2}$ are distinct finite complex numbers satisfying
$$(\beta_1-\alpha)^{2}\not=(\beta_2-\alpha)^{2}.$$
If two non-constant meromorphic functions $f$ and $g$ in $M_{1}(\mathbb{C})$ share $S_{1}$ IM, $S_2$ CM, and if the order of $f$ is neither an integer nor infinite, then $f\equiv g$.
\end{theoG}
A recent advent in the uniqueness theory of meromorphic functions is the introduction of the notion of weighted sharing instead CM sharing.
\begin{defi}(\cite{L})
  Let $l$ be a non-negative integer or infinity. For $a\in\mathbb{C}\cup\{\infty\}$, we denote by $E_{l}(a;f)$, the set of all $a$-points of $f$, where an $a$-point of multiplicity $m$ is counted $m$ times if $m\leq l$ and $l+1$ times if $m>l$. \par
If for two meromorphic functions $f$ and $g$, we have $E_{l}(a;f)=E_{l}(a;g)$, then we say that $f$ and $g$ share the value $a$ with weight $l$.
\end{defi}
The IM and CM sharing respectively correspond to weight $0$ and $\infty$.
\begin{defi}
For $S\subset \mathbb{C}\cup\{\infty\}$, we define $E_{f}(S,l)$ as $$E_{f}(S,l)=\displaystyle\bigcup_{a\in S}E_{l}(a;f),$$
where $l$ is a non-negative integer or infinity. Clearly $E_{f}(S)=E_{f}(S,\infty)$.\par
We say that $f$ and $g$ share $S$ with weight $l$, or simply $f$ and $g$ share $(S,l)$ if $E_{f}(S,l)=E_{g}(S,l)$.
\end{defi}
Regarding Theorems F and G, in (\cite{SK, SS}), the following question was asked:
\begin{ques}\label{bt3}(\cite{SK, SS})
It seems reasonable to conjecture that Theorems F and G still hold if  $f$ and $g$ share $(S_{1}, 2)$ (or, possibly, $(S_{1}, 0)$) and $S_{2}$ IM.
\end{ques}
In this direction, the next two theorems are given by P. Sahoo and H. Karmakar (\cite{SK}); and P. Sahoo and A. Sarkar (\cite{SS}) respectively. Before going to state their results, we need to recall the following definitions:
\begin{defi} (\cite{SK, SS})
Let $n$ be a positive integer and $S_{1} = \{\alpha_{1},\alpha_{2},\ldots,\alpha_{n}\}$, where $\alpha_{i}$'s are non-zero distinct complex constants. Suppose that
\bea\label{eqn1} P(z)=\frac{z^{n}-(\sum\alpha_{i})z^{n-1}+\ldots+(-1)^{n-1}(\sum \alpha_{i_{1}}\alpha_{i_{2}}\ldots\alpha_{i_{n-1}})z}{(-1)^{n+1}\alpha_{1}\alpha_{2}\ldots\alpha_{n}}.\eea
Let $m_{1}$ be the number of simple zeros of $P(z)$ and $m_{2}$ be the number of multiple zeros of $P(z)$. Then we define $\Gamma_{1} := m_{1}+m_{2}$
and $\Gamma_{2} := m_{1} + 2m_{2}$.
\end{defi}
\begin{defi}
  Let $$Q(z):=(z-\alpha_{1})(z-\alpha_{2})\ldots(z-\alpha_{n}).$$
  Then $Q(z)=(-1)^{n+1}\alpha_{1}\alpha_{2}\ldots\alpha_{n}\{P(z)-1\}$ and $Q'(z)=(-1)^{n+1}\alpha_{1}\alpha_{2}\ldots\alpha_{n}P'(z)$.\par
\end{defi}
Now, we state the results of Sahoo and Karmakar (\cite{SK}) and Sahoo and Sarkar (\cite{SS}) respectively.
\begin{theoH}(\cite{SK})
Let $f,g\in M_{1}(\mathbb{C})$ and $S_{1}=\{\alpha_1,\alpha_2,\ldots,\alpha_n\}$, $S_{2}=\{\beta_{1},\beta_{2}\}$, where $\alpha_1,\alpha_2,\ldots,\alpha_n, \beta_{1},\beta_{2}$ are $n+2$ distinct non-zero complex constants satisfying $n>2\Gamma_{2}.$ If $f$ and $g$ share $(S_{1},2)$ and $S_{2}$ IM, then $f\equiv g$, provided
$$(\beta_{1}-\alpha_{1})^{2}(\beta_{1}-\alpha_{2})^{2}\ldots(\beta_{1}-\alpha_{n})^{2}\not=(\beta_{2}-\alpha_{1})^{2}(\beta_{2}-\alpha_{2})^{2}\ldots(\beta_{2}-\alpha_{n})^{2}$$
and $f$ is of non-integer finite order.
\end{theoH}
\begin{theoI}(\cite{SS})
Let $f,g\in M_{1}(\mathbb{C})$ and $S_{1}=\{\alpha_1,\alpha_2,\ldots,\alpha_n\}$, $S_{2}=\{\beta_{1},\beta_{2}\}$, where $\alpha_1,\alpha_2,\ldots,\alpha_n, \beta_{1},\beta_{2}$ are $n+2$ distinct non-zero complex constants satisfying $n>2\Gamma_{2}+3\Gamma_{1}.$ If $f$ and $g$ share $S_{1}$ and $S_{2}$ IM, then $f\equiv g$, provided
$$(\beta_{1}-\alpha_{1})^{2}(\beta_{1}-\alpha_{2})^{2}\ldots(\beta_{1}-\alpha_{n})^{2}\not=(\beta_{2}-\alpha_{1})^{2}(\beta_{2}-\alpha_{2})^{2}\ldots(\beta_{2}-\alpha_{n})^{2}$$
and $f$ is of non-integer finite order.
\end{theoI}
\begin{rem}
  To prove the Theorems H and I, it was assumed in (\cite{SK, SS}) that the following two inequalities are true:
  \bea\label{bt1} N_{2}(r,0;P(f))\leq\Gamma_{2}\overline{N}(r,0;f),\eea
\bea\label{bt2} \overline{N}(r,0;P( f ))\leq \Gamma_{1}\overline{N}(r,0;f),\eea
where $N_{2}(r,0;f):=\overline{N}(r,0;f)+\overline{N}(r,0;f|\geq 2)$ and $\overline{N}(r,0;f|\geq 2)$ is the reduced counting function of those zeros of $f$ whose multiplicities are not less than $2$.\par
But the following example is countering the above two inequalities. It is known from (\cite{FR}) that the polynomial $\frac{(n-1)(n-2)}{2}z^{n}-n(n-2)z^{n-1}+\frac{n(n-1)}{2}z^{n-2}-c$, where $c (\not=0,1)\in \mathbb{C}$, $n(\geq 11)\in \mathbb{N}$, has $n$ distinct zeros. Let $$\mathcal{P}(z)=\frac{\frac{(n-1)(n-2)}{2}z^{n}-n(n-2)z^{n-1}+\frac{n(n-1)}{2}z^{n-2}}{c},$$ then
$$\mathcal{P}(z)=\frac{\frac{(n-1)(n-2)}{2}z^{n-2}(z-\gamma_{1})(z-\gamma_{2})}{c},$$
where $\gamma_{1},\gamma_{2}$ are the zeros of $z^{2}-\frac{2n}{n-1}z+\frac{n}{n-2}=0$. Here, we observed that the zeros of $\mathcal{P}(f)$ may come from $\gamma_{i}$-points of $f$. Thus the two inequalities (\ref{bt1}, \ref{bt2}) are not working always.
\end{rem}
The aim of this paper is to answer the question \ref{bt3} with the necessary modifications of those two inequalities (\ref{bt1}) and (\ref{bt2}).
\section{Main Results}
Let $n$ be a positive integer and $S_{1} = \{\alpha_{1},\alpha_{2},\ldots,\alpha_{n}\}$, where $\alpha_{i}$'s are non-zero distinct complex constants. Suppose that
\bea\label{eqn1} P(z)=\frac{z^{n}-(\sum\alpha_{i})z^{n-1}+\ldots+(-1)^{n-1}(\sum \alpha_{i_{1}}\alpha_{i_{2}}\ldots\alpha_{i_{n-1}})z}{(-1)^{n+1}\alpha_{1}\alpha_{2}\ldots\alpha_{n}}.\eea
Let $m_{1}$ be the number of simple zeros of $P(z)$ and $m_{2}$ be the number of multiple zeros of $P(z)$. Then we define $\Gamma_{1} := m_{1}+m_{2}$
and $\Gamma_{2} := m_{1} + 2m_{2}$. Further suppose that $P'(z)$ has $k$-distinct zeros.
\begin{theo}\label{thb1.1}
Let $f,g\in M_{1}(\mathbb{C})$ and $S_{1}=\{\alpha_1,\alpha_2,\ldots,\alpha_n\}$, $S_{2}=\{\beta_{1},\beta_{2}\}$, where $\alpha_1,\alpha_2,\ldots,\alpha_n, \beta_{1},\beta_{2}$ are $n+2$ distinct non-zero complex constants satisfying $n>\max\{2k+2, 2\Gamma_{1}$\}, where $k\geq 2$ and $\Gamma_{1}\geq 3$. If $f$ and $g$ share $(S_{1},2)$ and $S_{2}$ IM, then $f\equiv g$, provided
$$(\beta_{1}-\alpha_{1})(\beta_{1}-\alpha_{2})\ldots(\beta_{1}-\alpha_{n})\not=(\beta_{2}-\alpha_{1})(\beta_{2}-\alpha_{2})\ldots(\beta_{2}-\alpha_{n})$$
and $f$ is of non-integer finite order.
\end{theo}
\begin{theo}\label{thb1.2}
Let $f,g\in M_{1}(\mathbb{C})$ and $S_{1}=\{\alpha_1,\alpha_2,\ldots,\alpha_n\}$, $S_{2}=\{\beta_{1},\beta_{2}\}$, where $\alpha_1,\alpha_2,\ldots,\alpha_n, \beta_{1},\beta_{2}$ are $n+2$ distinct non-zero complex constants satisfying $n>\max\{2k+5, 2\Gamma_{1}$\}, where $k\geq 2$  and $\Gamma_{1}\geq 3$. If $f$ and $g$ share $S_{1}$ and $S_{2}$ IM, then $f\equiv g$, provided
$$(\beta_{1}-\alpha_{1})(\beta_{1}-\alpha_{2})\ldots(\beta_{1}-\alpha_{n})\not=(\beta_{2}-\alpha_{1})(\beta_{2}-\alpha_{2})\ldots(\beta_{2}-\alpha_{n})$$
and $f$ is of non-integer finite order.
\end{theo}
In Theorem E, one can observe that the elements of $S_{2}$ are the zeros of $\emph{p}'(z)$, where $$\emph{p}(z)=\frac{(n-1)(n-2)}{2}z^{n}-n(n-2)z^{n-1}+\frac{n(n-1)}{2}z^{n-2}+1.$$
This observation motivates us to write the next two theorems.
\begin{theo}\label{thb1.3}
Let $f,g\in M_{1}(\mathbb{C})$ and $S_{1}=\{\alpha_1,\alpha_2,\ldots,\alpha_n\}$, $S_{2}=\{\beta_{1},\beta_{2}\}$, where $\alpha_1,\alpha_2,\ldots,\alpha_n, \beta_{1},\beta_{2}$ are $n+2$ distinct non-zero complex constants satisfying $n>\max\{4, 2\Gamma_{1}$\}, and $P'(z)$ has exactly two distinct zeros $\beta_{1}$ and $\beta_{2}$ and $\Gamma_{1}\geq 3$. If $f$ and $g$ share $(S_{1},2)$ and $S_{2}$ IM, then $f\equiv g$, provided
$$(\beta_{1}-\alpha_{1})(\beta_{1}-\alpha_{2})\ldots(\beta_{1}-\alpha_{n})\not=(\beta_{2}-\alpha_{1})(\beta_{2}-\alpha_{2})\ldots(\beta_{2}-\alpha_{n})$$
and $f$ is of non-integer finite order.
\end{theo}
\begin{theo}\label{thb1.4}
Let $f,g\in M_{1}(\mathbb{C})$ and $S_{1}=\{\alpha_1,\alpha_2,\ldots,\alpha_n\}$, $S_{2}=\{\beta_{1},\beta_{2}\}$, where $\alpha_1,\alpha_2,\ldots,\alpha_n, \beta_{1},\beta_{2}$ are $n+2$ distinct non-zero complex constants satisfying $n>\max\{7, 2\Gamma_{1}$\}, and $P'(z)$ has exactly two distinct zeros $\beta_{1}$ and $\beta_{2}$ and $\Gamma_{1}\geq 3$. If $f$ and $g$ share $S_{1}$ and $S_{2}$ IM, then $f\equiv g$, provided
$$(\beta_{1}-\alpha_{1})(\beta_{1}-\alpha_{2})\ldots(\beta_{1}-\alpha_{n})\not=(\beta_{2}-\alpha_{1})(\beta_{2}-\alpha_{2})\ldots(\beta_{2}-\alpha_{n})$$
and $f$ is of non-integer finite order.
\end{theo}
\section{Necessary Lemmas}
\begin{lem}\label{lem1} (\cite{2C, Ch})
Let $f,g\in M(\mathbb{C})$ and $f$ and $g$ share the set $\{\beta_{1},\beta_{2}\}$ IM, where $\beta_{1}\not=\beta_{2}$ and $\beta_{1},\beta_{2}\in\mathbb{C}$. Then $\rho(f)=\rho(g)$.
\end{lem}
\begin{lem}\label{lem2}
Let $f,g$ be two non-constant meromorphic functions and $a_1,a_2$ be two distinct finite complex numbers. If $f$ and $g$ share $a_{1}$, $a_{2}$ and $\infty$ CM, then $f\equiv g$, provided that $f$ is of non-integer finite order.
\end{lem}
\begin{proof}
  The proof follows from Lemma \ref{lem1} and Theorem 2.19 of (\cite{YY}).
\end{proof}
\section{Proof of the theorems}
\begin{proof}[\textbf{Proof of the Theorem \ref{thb1.1}}]
Given that $f,g\in M_{1}(\mathbb{C})$. Thus
\bea\label{neq1} N(r,\infty;f)=O(\ln r),  N(r,\infty;g)=O(\ln r).\eea
Now, we put $$Q(z)=(z-\alpha_{1})(z-\alpha_{2})\ldots(z-\alpha_{n}),$$
and
$$F(z):=\frac{1}{Q(f(z))}~~\text{and}~~G(z):=\frac{1}{Q(g(z))}.$$
Thus from equation (\ref{eqn1}), we have $(-1)^{n+1}\alpha_{1}\alpha_{2}\ldots \alpha_{n}P'(z)=Q'(z)$. Let $S(r):(0,\infty)\rightarrow\mathbb{R}$ be any function satisfying $S(r)=o(T(r,F)+T(r,G))$ for $r\rightarrow\infty$ outside a set of finite Lebesgue measure.\par
Let
$$H(z):=\frac{F''(z)}{F'(z)}-\frac{G''(z)}{G'(z)}.$$
Now, we consider two cases:\par
\textbf{Case-I} First we assume that $H\not\equiv 0$. Since $H(z)$ can be expressed as
$$H(z)=\frac{G'(z)}{F'(z)}\left(\frac{F'(z)}{G'(z)}\right)',$$
so all poles of $H$ are simple. Also, \textbf{poles of $H$ may occur} at
\begin{enumerate}
 \item poles of $F$ and $G$,
  \item zeros of $F'$ and $G'$.
\end{enumerate}
But using the Laurent series expansion of $H$, it is clear that \enquote{simple poles} of $F$ (hence, that of $G$) are the zeros of $H$. Thus
\bea \label{equn1.1} N(r,\infty;F|=1)=N(r,\infty;G|=1)\leq N(r,0;H), \eea
where $N(r,\infty;F|=1)$ is the the counting function of simple poles of $F$.
Using the lemma of logarithmic derivative and the first fundamental theorem, (\ref{equn1.1}) can be written as
\bea \label{equn1.2} N(r,\infty;F|=1)=N(r,\infty;G|=1)\leq N(r,\infty;H)+S(r). \eea
Let $\lambda_{1},\lambda_{2},\ldots,\lambda_{k}$ be the $k$- distinct zeros of $P'(z)$. Since $$F'(z)=-\mu\frac{f'(z)P'(f(z))}{(Q(f(z)))^{2}},~~ G'(z)=-\mu\frac{g'(z)P'(g(z))}{(Q(g(z)))^{2}},$$ where $\mu=(-1)^{n+1}\alpha_{1}\alpha_{2}\ldots \alpha_{n}$, and $f$, $g$ share $(S_{1},2)$, by simple calculations, we can write
\bea \label{equn1.3} N(r,\infty;H)&\leq& \sum_{j=1}^{k}\left(\overline{N}(r,\lambda_{j};f)+\overline{N}(r,\lambda_{j};g)\right)+\ol{N}_{0}(r,0;f')+\ol{N}_{0}(r,0;g') \\
\nonumber &&+ \ol{N}(r,\infty;f)+\ol{N}(r,\infty;g)+\ol{N}_{\ast}(r,\infty;F,G),\eea
$\ol{N}_{0}(r,0;f')$ denotes the reduced counting function of zeros of $f'$, which are not zeros of $\prod_{i=1}^{n}(f-\alpha_{i})\prod_{j=1}^{k}(f-\lambda_{j})$. Similarly $\ol{N}_{0}(r,0;g')$ is defined. Also, $\ol{N}_{\ast}(r,\infty;F,G)$ is the  reduced counting function of those poles of $F$ whose multiplicities differ from the multiplicities of the corresponding poles of $G$.\par
Now, using the second fundamental theorem to the functions $f$ and $g$, we obtain
\bea \label{equn1.444} && (n+k-1)\left(T(r,f)+T(r,g)\right)\\
\nonumber &\leq& \ol{N}(r,\infty;f)+\sum_{i=1}^{n}\ol{N}(r,\alpha_{i};f)+\sum_{j=1}^{k}\overline{N}(r,\lambda_{j};f)-\ol{N}_{0}(r,0;f')\\
\nonumber &&+\ol{N}(r,\infty;g)+\sum_{i=1}^{n}\ol{N}(r,\alpha_{i};g)+\sum_{j=1}^{k}\overline{N}(r,\lambda_{j};g)-\ol{N}_{0}(r,0;g')\\
\nonumber &&+S(r,f)+S(r,g).\eea
Now, using the inequalities (\ref{equn1.2}), (\ref{equn1.3}) and (\ref{equn1.444}), we have
\bea \label{equn1.4} && (n+k-1)\left(T(r,f)+T(r,g)\right)\\
\nonumber &\leq& \ol{N}(r,\infty;f)+\ol{N}(r,\infty;g)+\ol{N}(r,\infty;F)+\ol{N}(r,\infty;G)\\
\nonumber &+&\sum_{j=1}^{k}\left(\overline{N}(r,\lambda_{j};f)+\overline{N}(r,\lambda_{j};g)\right)-\ol{N}_{0}(r,0;f')-\ol{N}_{0}(r,0;g')+S(r)\\
\nonumber &\leq& 2\left(\ol{N}(r,\infty;f)+\ol{N}(r,\infty;g)\right)+2\sum_{j=1}^{k}\left(\overline{N}(r,\lambda_{j};f)+\overline{N}(r,\lambda_{j};g)\right)\\
\nonumber &+&\ol{N}(r,\infty;F|\geq2)+\ol{N}(r,\infty;G)+\ol{N}_{\ast}(r,\infty;F,G)+S(r).\eea
Noting that
\beas \ol{N}(r,\infty;F)-\frac{1}{2}N(r,\infty;F|=1)+\frac{1}{2}\ol{N}_{\ast}(r,\infty;F,G)\leq \frac{1}{2}N(r,\infty;F),\\
\ol{N}(r,\infty;G)-\frac{1}{2}N(r,\infty;G|=1)+\frac{1}{2}\ol{N}_{\ast}(r,\infty;F,G)\leq \frac{1}{2}N(r,\infty;G).
\eeas
Thus (\ref{equn1.4}) can be written as
\bea \label{equn1.5} && (n+k-1)\left(T(r,f)+T(r,g)\right)\\
\nonumber &\leq& 2\left(\ol{N}(r,\infty;f)+\ol{N}(r,\infty;g)\right)+(2k+\frac{n}{2})(T(r,f)+T(r,g))+S(r)\\
\nonumber &\leq& (2k+\frac{n}{2})(T(r,f)+T(r,g))+O(\ln r)+S(r),\eea
which contradicts the assumption $n>2k+2$. Thus $H\equiv 0$.\par
\textbf{Case-II} Next, we assume that $H\equiv 0$. Then by integration, we have
\bea\label{es1} \frac{1}{Q(f(z))}&\equiv&\frac{c_0}{Q(g(z))}+c_{1},\eea
where $c_{0}$ is a non zero complex constant. Thus
$$T(r,f)=T(r,g)+O(1).$$
Now, we consider two cases:\\
\textbf{Subcase-I} First we assume that $c_{1}\not=0.$ Then equation (\ref{es1}) can be written as
$$Q(f)\equiv \frac{Q(g)}{c_{1}Q(g)+c_0}.$$
Thus $$\overline{N}(r,-\frac{c_0}{c_1};Q(g))=\overline{N}(r,\infty;Q(f))=O(\ln r).$$
We know that $Q(z)+\mu=\mu P(z)$. Now, if $\mu\not=\frac{c_0}{c_1}$, then applying the second fundamental theorem to $Q(g)$, we obtain
\beas &&n T(r,g)+O(1)\\
&=&T\left(r, Q(g)\right)\\
&\leq& \overline{N}\left(r,\infty;Q(g)\right)+\overline{N}\left(r,-\mu;Q(g)\right)+\overline{N}\left(r,-\frac{c_0}{c_1};Q(g)\right)+S(r,Q(g))\\
&\leq& \overline{N}\left(r,0;P(g)\right)+O(\ln r)+S(r,g)\\
&\leq& \Gamma_{1}T(r,g)+O(\ln r)+S(r,g),\eeas
which is impossible as $n> 2\Gamma_{1}$. Thus $\mu=\frac{c_0}{c_1}$. Hence
$$Q(f)\equiv \frac{Q(g)}{c_{0}P(g)}.$$
Since $P(z)$ has $m_{1}$ simple zeros and $m_{2}$ multiple zeros, so we can assume
$$P(z)=a_{0}(z-b_{1})(z-b_{2})\ldots(z-b_{m_{1}})(z-c_{1})^{l_{1}}(z-c_{2})^{l_{2}}\ldots(z-c_{m_{2}})^{l_{m_{2}}},$$
where $l_{i}\geq 2$ for $1\leq i\leq m_{2}$. Thus every zero of $g-b_{j}$ ($1\leq j \leq m_{1}$) has a multiplicity atleast $n$. As $P'(z)$ has at least two zeros, so, $l_{i}< n$, and hence each zero of $g-c_{i}$ ($1\leq i \leq m_{2}$) has a multiplicity atleast $2$.\par Thus applying the second fundamental theorem to $g$, we have
\beas
&&(m_{1}+m_{2}-1)T(r,g)\\
&\leq& \overline{N}(r,\infty;g)+\sum\limits_{j=1}^{m_1}\overline{N}(r,b_{j};g)+\sum\limits_{i=1}^{m_2}\overline{N}(r,c_{i};g)+S(r,g)\\
&\leq& \frac{1}{n}\sum\limits_{j=1}^{m_1}N(r,b_{j};g)+\frac{1}{2}\sum\limits_{i=1}^{m_2}N(r,c_{i};g)+O(\ln r)+S(r,g)\\
&\leq& \frac{m_1}{n}T(r,g)+\frac{m_{2}}{2}T(r,g)+O(\ln r)+S(r,g)\\
&\leq& \frac{m_1}{2}T(r,g)+\frac{m_{2}}{2}T(r,g)+O(\ln r)+S(r,g)\\
&\leq& \frac{m_1+m_2}{2}T(r,g)+O(\ln r)+S(r,g),
\eeas
which is impossible as $\Gamma_{1}\geq 3$.\\
\textbf{Subcase-II} Next we assume that $c_{1}=0$. Then equation (\ref{es1}) can be written as
$$Q(g)\equiv c_{0} Q(f).$$ Thus
$$P(g)\equiv c_{0}(P(f)-1+\frac{1}{c_0}).$$
If $c_{0}\not=1$, then using the second fundamental theorem to $P(f)$, we obtain
\beas &&n T(r,f)+O(1)\\
&=&T\left(r, P(f)\right)\\
&\leq& \overline{N}\left(r,\infty;P(f)\right)+\overline{N}\left(r,0;P(f)\right)+\overline{N}\left(r,1-\frac{1}{c_{0}};P(f)\right)+S(r,f)\\
&\leq& \overline{N}\left(r,0;P(f)\right)+\overline{N}\left(r,0;P(g)\right)+O(\ln r)+S(r,f)\\
&\leq& \Gamma_{1}T(r,f)+\Gamma_{1}T(r,g)+O(\ln r)+S(r,f),
\eeas
which is impossible as $n> 2\Gamma_{1}$. Thus $c_0=1$, i.e.,
$$Q(f)\equiv Q(g).$$
Thus $$(f-\alpha_1)(f-\alpha_2)\ldots(f-\alpha_n)\equiv (g-\alpha_1)(g-\alpha_2)\ldots(g-\alpha_n).$$
Given that $f$ and $g$ share $\{\beta_{1},\beta_{2}\}$ IM and $$(\beta_{1}-\alpha_{1})(\beta_{1}-\alpha_{2})\ldots(\beta_{1}-\alpha_{n})\not=(\beta_{2}-\alpha_{1})(\beta_{2}-\alpha_{2})\ldots(\beta_{2}-\alpha_{n}).$$ Thus, if $z_{0}$ be a $\beta_{1}$ point of $f$, then $z_{0}$ can't be a $\beta_{2}$ point of $g$. Thus $f$ and $g$ share $\beta_1$ and $\beta_2$ IM.\par
Consequently, we can say that $f$ and $g$ share $\beta_1$, $\beta_2$ and $\infty$ CM. Thus by Lemma \ref{lem2}, $f\equiv g$. This completes the proof.
\end{proof}
\begin{proof}[\textbf{Proof of the Theorem \ref{thb1.2}}]
First we observe that if $F$ and $G$ share $\infty$ IM, then
\beas &&\ol{N}(r,\infty;F)+\ol{N}(r,\infty;G)-N(r,\infty;F|=1)-\frac{1}{2}\ol{N}_{\ast}(r,\infty;F,G)\\
&&\leq \frac{1}{2}\left(N(r,\infty;F)+N(r,\infty;G)\right).
\eeas
If $F$ and $G$ are defined as in the proof of Theorem \ref{thb1.1}, then
\beas
&&\ol{N}_{\ast}(r,\infty;F,G)\\
&\leq& N(r,\infty;F)-\overline{N}(r,\infty;F)+N(r,\infty;G)-\overline{N}(r,\infty;G)\\
&\leq& N(r,0;Q(f))-\overline{N}(r,0;Q(f))+N(r,0;Q(g))-\overline{N}(r,0;Q(g))\\
&\leq& N(r,0;f'|f\not=0)+N(r,0;g'|g\not=0)\\
&\leq& N(r,0;\frac{f'}{f})+N(r,0;\frac{g'}{g})\\
&\leq& N(r,\infty;\frac{f'}{f})+N(r,\infty;\frac{g'}{g})+S(r,f)+S(r,g)\\
&\leq& N(r,\infty;f)+N(r,0;f)+N(r,\infty;g)+N(r,0;g)+S(r,f)+S(r,g).\eeas
Thus proceeding similarly as Case-I of Theorem \ref{thb1.1}, (\ref{equn1.4}) can be written as
\bea \label{equn1.5} && (n+k-1)\left(T(r,f)+T(r,g)\right)\\
\nonumber &\leq& \frac{7}{2}\left(\ol{N}(r,\infty;f)+\ol{N}(r,\infty;g)\right)+(2k+\frac{n}{2}+\frac{3}{2})(T(r,f)+T(r,g))+S(r)\\
\nonumber &\leq& (2k+\frac{n}{2}+\frac{3}{2})(T(r,f)+T(r,g))+O(\ln r)+S(r),\eea
which contradicts the assumption $n>2k+5$.\\
Remaining part of the proof is similar to the case-II of Theorem \ref{thb1.1}.
\end{proof}
\begin{proof}[\textbf{Proof of the Theorem \ref{thb1.3}}]
Given $f,g\in M_{1}(\mathbb{C})$. Thus
\bea\label{neq1} N(r,\infty;f)=O(\ln r),  N(r,\infty;g)=O(\ln r).\eea
Now, we put $$Q(z)=(z-\alpha_{1})(z-\alpha_{2})\ldots(z-\alpha_{n}),$$
and
$$F(z):=\frac{1}{Q(f(z))}~~\text{and}~~G(z):=\frac{1}{Q(g(z))}.$$
Further suppose that
$$P'(z)=b_{0}(z-\beta_{1})^{q_1}(z-\beta_{2})^{q_2}.$$
Thus $S_{1}=\{\alpha_{1},\alpha_{2},\ldots,\alpha_{n}\}$ and $S_{2}=\{\beta_{1},\beta_{2}\}$. Since $f$ and $g$ share $S_{2}$ IM, so
$$\sum_{j=1}^{2}\overline{N}(r,\beta_{j};f)=\sum_{j=1}^{2}\overline{N}(r,\beta_{j};g).$$
Let $S(r):(0,\infty)\rightarrow\mathbb{R}$ be any function satisfying $S(r)=o(T(r,F)+T(r,G))$ for $r\rightarrow\infty$ outside a set of finite Lebesgue Measure.\par
Let
$$H(z):=\frac{F''(z)}{F'(z)}-\frac{G''(z)}{G'(z)}.$$
Now, we consider two cases:\par
\textbf{Case-I} First we assume that $H\not\equiv 0$. Since $H(z)$ can be expressed as
$$H(z)=\frac{G'(z)}{F'(z)}\left(\frac{F'(z)}{G'(z)}\right)',$$
so all poles of $H$ are simple. Also, \textbf{poles of $H$ may occur} at
\begin{enumerate}
 \item poles of $F$ and $G$,
  \item zeros of $F'$ and $G'$.
\end{enumerate}
But using the Laurent series expansion of $H$, it is clear that \enquote{simple poles} of $F$ (hence, that of $G$) is a zero of $H$. Thus
\bea \label{equnb1.1} N(r,\infty;F|=1)=N(r,\infty;G|=1)\leq N(r,0;H). \eea
Using the lemma of logarithmic derivative and the first fundamental theorem, (\ref{equnb1.1}) can be written as
\bea \label{equnb1.2} N(r,\infty;F|=1)=N(r,\infty;G|=1)\leq N(r,\infty;H)+S(r). \eea
Since $$F'(z)=-\mu\frac{f'(z)P'(f(z))}{(Q(f(z)))^{2}},~~ G'(z)=-\mu\frac{g'(z)P'(g(z))}{(Q(g(z)))^{2}},$$ where $\mu=(-1)^{n+1}\alpha_{1}\alpha_{2}\ldots \alpha_{n}$, and $f$, $g$ share $(S_1,2)$ and $(S_2,0)$, by simple calculations, we can write
\bea \label{equnb1.3} N(r,\infty;H)&\leq& \sum_{j=1}^{2}\overline{N}(r,\beta_{j};f)+\ol{N}_{0}(r,0;f')+\ol{N}_{0}(r,0;g') \\
\nonumber &&+ \ol{N}(r,\infty;f)+\ol{N}(r,\infty;g)+\ol{N}_{\ast}(r,\infty;F,G),\eea
where $\ol{N}_{0}(r,0;f')$ denotes the reduced counting function of zeros of $f'$, which are not zeros of $\prod_{i=1}^{n}(f-\alpha_{i})\prod_{j=1}^{2}(f-\beta_{j})$. Similarly, $\ol{N}_{0}(r,0;g')$ is defined.
Now, using the second fundamental theorem and (\ref{equnb1.2}), (\ref{equnb1.3}), we have
\bea \label{equnb1.4} && (n+1)\left(T(r,f)+T(r,g)\right)\\
\nonumber &\leq& \ol{N}(r,\infty;f)+\sum_{i=1}^{n}\overline{N}(r,\alpha_{i};f)+\sum_{j=1}^{2}\overline{N}(r,\beta_{j};f)-\ol{N}_{0}(r,0;f')\\
\nonumber &&+\ol{N}(r,\infty;g)+\sum_{i=1}^{n}\overline{N}(r,\alpha_{i};g)+\sum_{j=1}^{2}\overline{N}(r,\beta_{j};g)-\ol{N}_{0}(r,0;g')\\
\nonumber &&+S(r,f)+S(r,g)\\
\nonumber &\leq& 2\left(\ol{N}(r,\infty;f)+\ol{N}(r,\infty;g)\right)+\sum_{j=1}^{2}\left(2\overline{N}(r,\beta_{j};f)+\overline{N}(r,\beta_{j};g)\right)\\
\nonumber &+&\ol{N}(r,\infty;F|\geq2)+\ol{N}(r,\infty;G)+\ol{N}_{\ast}(r,\infty;F,G)+S(r).\eea
Noting that
\beas \ol{N}(r,\infty;F)-\frac{1}{2}N(r,\infty;F|=1)+\frac{1}{2}\ol{N}_{\ast}(r,\infty;F,G)\leq \frac{1}{2}N(r,\infty;F),\\
\ol{N}(r,\infty;G)-\frac{1}{2}N(r,\infty;G|=1)+\frac{1}{2}\ol{N}_{\ast}(r,\infty;F,G)\leq \frac{1}{2}N(r,\infty;G).
\eeas
Thus (\ref{equnb1.4}) can be written as
\bea \label{equnb1.5} && (n+1)\left(T(r,f)+T(r,g)\right)\\
\nonumber &\leq& 2\left(\ol{N}(r,\infty;f)+\ol{N}(r,\infty;g)\right)+(3+\frac{n}{2})(T(r,f)+T(r,g))+S(r)\\
\nonumber &\leq& (3+\frac{n}{2})(T(r,f)+T(r,g))+O(\ln r)+S(r),\eea
which contradicts the assumption $n> 4$. Thus $H\equiv 0$.\par
\textbf{Case-II} Next we assume that $H\equiv 0$. Remaining part of the proof is similar to the proof of case-II of Theorem \ref{thb1.1}.
\end{proof}
\begin{proof}[\textbf{Proof of the Theorem \ref{thb1.4}}] The idea of the proof follows from the proof of Theorem \ref{thb1.3} and Theorem \ref{thb1.2}.
\end{proof}
\begin{center} {\bf Acknowledgement} \end{center}
The authors are grateful to the anonymous referees for their valuable suggestions which considerably improved the presentation of the paper.\par
The research work of the first and the fourth authors are  supported by the Department of Higher Education, Science and Technology \text{\&} Biotechnology, Govt. of West Bengal under the sanction order no. 216(sanc) /ST/P/S\text{\&}T/16G-14/2018 dated 19/02/2019.\par
The second and the third authors are thankful to the Council of Scientific and Industrial Research, HRDG, India for granting Junior Research
Fellowship (File No.: 09/106(0179)/2018-EMR-I and 08/525(0003)/2019-EMR-I respectively) during the tenure of which this work was done.

\end{document}